\documentclass[12pt]{article}
\usepackage{amsfonts,amssymb,amsmath,amsthm}
\usepackage{graphicx}

\date{}
\newtheorem{theorem}{Theorem}[section]
\newtheorem{proposition}{Proposition}[section]
\newtheorem{lemma}{Lemma}[section]

\theoremstyle{definition}
\newtheorem{definition}{Definition}[section]

\numberwithin{equation}{section}

\newcommand{\R}{\mathbb{R}}
\newcommand{\N}{\mathbb{N}}
\newcommand{\Z}{\mathbb{Z}}
\newcommand{\Q}{\mathbb{Q}}
\newcommand{\T}{\mathbb{T}}
\newcommand{\D}{\mathcal{D}}
\newcommand{\esslim}{\operatornamewithlimits{ess\,lim}}
\newcommand{\esslimsup}{\operatornamewithlimits{ess\,lim\,sup}}
\newcommand{\sgn}{\operatorname{sign}}
\newcommand{\meas}{\operatorname{meas}}
\newcommand{\Cl}{\operatorname{Cl}}
\newcommand{\Int}{\operatorname{Int}}
\newcommand{\pr}{\operatorname{pr}}
\newcommand{\supp}{\operatorname{supp}}
\newcommand{\const}{\mathrm{const}}
\renewcommand{\div}{\operatorname{div}}
\begin{document}

\title{On decay of entropy solutions to degenerate nonlinear parabolic equations with perturbed periodic initial data}
\author{Evgeny Yu. Panov}

\maketitle

\begin{abstract}
Under a precise nonlinearity-diffusivity assumption we establish the decay of entropy solutions of a degenerate nonlinear parabolic equation with initial data being a sum of periodic function and a function vanishing at infinity (in the sense of measure).
\end{abstract}

%

\section{Introduction}
In the half-space $\Pi=\R_+\times\R^n$, $\R_+=(0,+\infty)$, we consider the nonlinear parabolic equation
\begin{equation}\label{1}
u_t+\div_x(\varphi(u)-a(u)\nabla u)=0,
\end{equation}
where the flux vector $\varphi(u)=(\varphi_1(u),\ldots,\varphi_n(u))$ is merely continuous: $\varphi_i(u)\in C(\R)$, $i=1,\ldots,n$, and the diffusion matrix $a(u)=(a_{ij}(u))_{i,j=1}^n$ is Lebesgue measurable and bounded:
$a_{ij}(u)\in L^\infty(\R)$, $i,j=1,\ldots,n$. We also assume that the matrix ${a(u)\ge 0}$
(nonnegative definite). This matrix may have nontrivial kernel. Hence (\ref{1}) is a degenerate
(hyperbolic-parabolic) equation. In particular case $a\equiv 0$ it reduces to a first order conservation law
 \begin{equation}\label{con}
u_t+\div_x \varphi(u)=0.
\end{equation}
Equation (\ref{1}) can be written (at least formally) in the conservative form
\begin{equation}\label{1c}
u_t+\div_x\varphi(u)-D^2_x\cdot A(u)=0,
\end{equation}
where the matrix $A(u)$ is a primitive of the matrix $a(u)$, $A'(u)=a(u)$, and the operator $D^2_x$ is the second order  ``divergence''
$$
D^2_x\cdot A(u)\doteq\sum_{i,j=1}^n \frac{\partial^2}{\partial x_i\partial x_j}A_{ij}(u), \quad u=u(t,x).
$$
Equation (\ref{1}) is endowed with the initial condition
\begin{equation}\label{2}
u(0,x)=u_0(x)\in L^\infty(\R^n).
\end{equation}
Let $g(u)\in BV_{loc}(\R)$ be a function of bounded variation on any segment in $\R$.
We will need the bounded linear operator $T_g:C(\R)/C\to C(\R)/C$, where $C$ is the space of constants. This operator is defined up to an additive constant by the relation
\begin{equation}\label{efl}
T_g(f)(u)=g(u-)f(u)-\int_0^u f(s)dg(s),
\end{equation}
where
$\displaystyle g(u-)=\lim_{v\to u-} g(v)$ is the left limit of $g$ at the point $u$, and the integral in (\ref{efl}) is understood in accordance with the formula
$$
\int_0^u f(s)dg(s)=\sgn u\int_{J(u)} f(s)dg(s),
$$
where $\sgn u=1$,
$J(u)$ is the interval $[0,u)$ if $u>0$, and $\sgn u=-1$, $J(u)=[u,0)$ if $u\le 0$.
Observe that $T_g(f)(u)$ is continuous even in the case of discontinuous $g(u)$. For instance, if $g(u)=\sgn(u-k)$ then
$T_g(f)(u)=\sgn(u-k)(f(u)-f(k))$. Notice also that for $f\in C^1(\R)$ the operator $T_g$ is uniquely determined by the identity $T_g(f)'(u)=g(u)f'(u)$ (in $\D'(\R)$).

We fix some representation of the diffusion matrix $a(u)$ in the form $a(u)=b^\top(u)b(u)$, where $b(u)=(b_{ij}(u))$, $i=1,\ldots,r$, $j=1,\ldots,n$ is a $r\times n$ matrix-valued function (a square root of $a(u)$) with measurable and bounded entries,  $b_{ij}(u)\in L^\infty(\R)$.
We recall the notion of entropy solution of the Cauchy problem (\ref{1}), (\ref{2}) introduced in \cite{ChPer1}.

\begin{definition}\label{def1}
A function $u=u(t,x)\in L^\infty(\Pi)$ is called an entropy solution (e.s. for short) of (\ref{1}), (\ref{2}) if
the following conditions hold:

(i) for each $i=1,\ldots,r$ the distributions
\begin{equation}\label{pr}
\div_x B_i(u(t,x))\in L^2_{loc}(\Pi),
\end{equation}
where vectors $B_i(u)=(B_{i1}(u),\ldots,B_{in}(u))\in C(\R,\R^n)$, and $B_{ij}'(u)=b_{ij}(u)$, $i=1,\ldots,r$, $j=1,\ldots,n$;

(ii) for every $g(u)\in C^1(\R)$, $i=1,\ldots,r$
\begin{equation}\label{cr}
\div_x T_g(B_i)(u(t,x))=g(u(t,x))\div_x B_i(u(t,x)) \ \mbox{ in } \D'(\Pi);
\end{equation}

(iii) for any convex function $\eta(u)\in C^2(\R)$
\begin{equation}\label{entr}
\eta(u)_t+\div_x T_{\eta'}(\varphi)(u)-D^2_x\cdot T_{\eta'}(A)(u)+\eta''(u)\sum_{i=1}^r (\div_x B_i(u))^2\le 0 \ \mbox{ in } \D'(\Pi);
\end{equation}

(iv) $\displaystyle\esslim_{t\to 0} u(t,\cdot)=u_0$ in $L^1_{loc}(\R^n)$.
\end{definition}

In the isotropic case when the diffusion matrix is scalar the definition can be considerably simplified and was introduced earlier by J.~Carrillo in \cite{Car}.

Relation (\ref{entr}) means that for any non-negative test function $f=f(t,x)\in C_0^\infty(\Pi)$
\begin{equation}\label{entrI}
\int_\Pi \bigl[\eta(u)f_t+T_{\eta'}(\varphi)(u)\cdot\nabla_x f+T_{\eta'}(A)(u)\cdot D^2_x f-f\eta''(u)\sum_{i=1}^r (\div_x B_i(u))^2\bigr]dtdx\ge 0,
\end{equation}
where $D^2_x f$ is the symmetric matrix of second order derivatives of $f$, and "$\cdot$" denotes the standard scalar multiplications of vectors or matrices (in particular, $A\cdot B=\mathop{\rm Tr} A^\top B$ for matrices $A,B$).

In the case of conservation laws (\ref{con}) Definition~\ref{def1} reduces to the known definition of entropy solutions
in the sense of S.\,N.~Kruzhkov \cite{Kr}. Taking in (\ref{entr}) $\eta(u)=\pm u$, we deduce that
$$
u_t+\div_x \varphi(u)-D^2_x\cdot A(u)=0 \ \mbox{ in } \D'(\Pi),
$$
that is, an e.s. $u$ is a weak solution of (\ref{1c}).

It is known that e.s. of (\ref{1}), (\ref{2}) always exists but in the case of only continuous flux may be nonunique, for conservations laws (\ref{con}) the corresponding  examples can be found in \cite{KrPa1,KrPa2}. Nevertheless, if initial function is periodic, the uniqueness holds:
an e.s. of (\ref{1}), (\ref{2}) is unique and space-periodic, see \cite[Theorem~1.3]{PaJDE2}. In general case there always exists the unique largest and smallest e.s., see \cite[Theorem~1.1]{PaJDE2}.

\section{Formulation of the results}
In the present paper we study the long time decay property of e.s. in the case when initial data is a perturbed periodic function. More precisely, we assume that the initial function $u_0(x)=p(x)+v(x)$, where $p(x)$ is periodic while $v(x)\in L_0^\infty(\R^n)$, where the space $L_0^\infty(\R^n)$ consists of functions $v(x)\in L^\infty(\R^n)$ such that
\begin{equation}\label{van}
 \forall\lambda>0 \quad \meas\{ \ x\in\R^n: \ |v(x)|>\lambda \ \}<+\infty
\end{equation}
(here we denote by $\meas$ the Lebesgue measure on $\R^n$). Evidently, the space $L_0^\infty(\R^n)$ contains functions vanishing at infinity as well as functions from the spaces $L^q(\R^n)\cap L^\infty(\R^n)$, $q>0$.
Observe that the functions $p,v$ are uniquely defined (up to equality on a set of full measure) by the function $u_0$.
Let
\begin{equation}\label{per}
G=\{ \ e\in\R^n \ | \ p(x+e)=p(x) \ \mbox{ almost everywhere in } \R^n \ \}
\end{equation}
be the group of periods of $p$, it is not necessarily a lattice because $p(x)$ may be constant in some directions. For example, if $p\equiv\const$ then $G=\R^n$. The periodicity of $p$ means that the linear hull of $G$ coincides with $\R^n$, that is, there is a basis of periods of $p$. Denote by $H$ the maximal linear subspace contained in $G$. The dual lattice
$$
G'=\{ \ \xi\in\R^n \ | \ \xi\cdot e\in\Z \ \forall e\in G \ \}
$$
is indeed a lattice in the orthogonal complement $H^\perp$ of the space $H$ (we will prove this simple statement in Lemma~\ref{lem2} below). Observe also that $G'=L_0'$ in $H^\perp$, where
$L_0=G\cap H^\perp$, so that $G=H\oplus L_0$. It is rather well-known (at least for continuous periodic functions) that $L_0$ is a lattice in $H^\perp$. The case of measurable periodic functions requires some little modifications and, for the sake of completeness, we put the proof of this fact in Lemma~\ref{lem2}. Notice that we use more general notion of period contained in (\ref{per}). For the standard notion $p(x+e)\equiv p(x)$ (where the words ``almost everywhere'' are omitted) the group $G$ may have more complicate structure. For example, the group of periods of the Dirichlet function on $\R$ is a set of rationals $\Q$, which is not a lattice in $\R$. We introduce the torus $\T^d=\R^n/G=H^\perp/L_0$
of dimension $d=\dim H^\perp=n-\dim H$ equipped with the normalized Lebesgue measure $dy$. The periodic function $p$ can be considered as a function on this torus $\T^d$: $p=p(y)$. Let
$$m=\int_{\T^d} p(y)dy$$ be the mean value of this function. Clearly, this value coincides with the mean value of the initial data:
$$
m=\lim_{R\to\infty}\frac{1}{|B_R|}\int_{B_R} u_0(x)dx,
$$
where $B_R$ is the ball $|x|<R$, and $|B_R|\doteq\meas B_R$ is its Lebesgue measure. The latter follows from the fact that functions $v(x)$ from $L_0^\infty(\R^n)$ always have zero mean value. More precisely,
\begin{equation}\label{mean0}
\lim_{|A|\to\infty}\frac{1}{|A|}\int_A |v(x)|dx=0,
\end{equation}
where $A$ runs over Lebesgue measurable sets of finite measure $|A|$. In fact, let $\varepsilon>0$, $E=\{ \ x\in\R^n: \ |v(x)|>\varepsilon \ \}$. Then $p=\meas E<+\infty$.
Obviously
$$
\int_A |v(x)|dx=\int_E |v(x)|dx+\int_{A\setminus E} |v(x)|dx\le p\|v\|_\infty+\varepsilon |A|.
$$
This implies that
$$
\frac{1}{|A|}\int_A |v(x)|dx\le \frac{p\|v\|_\infty}{|A|}+\varepsilon\mathop{\to}_{|A|\to\infty}\varepsilon.
$$
Hence,
$$
\limsup_{|A|\to\infty}\frac{1}{|A|}\int_A |v(x)|dx\le\varepsilon,
$$
and since $\varepsilon>0$ is arbitrary, we conclude that
$$
\lim_{|A|\to\infty}\frac{1}{|A|}\int_A |v(x)|dx=0,
$$
as was to be proved.

We will study the long time decay property of e.s. with respect to the following shift-invariant Stepanov norm on $L^\infty(\R^n)$:
 \begin{equation}\label{normX}
\|u\|_X=\sup_{y\in\R^n} \int_{|x-y|<1} |u(x)|dx
\end{equation}
(where we denote by $|z|$ the Euclidean norm of a finite-dimensional vector $z$).
As was demonstrated in \cite{PaSIMA2}, this norm is equivalent to each of more general norms
\begin{equation}\label{normV}
\|u\|_V=\sup_{y\in\R^n} \int_{y+V} |u(x)|dx,
\end{equation}
where $V$ is any bounded open set in $\R^n$ (the original norm $\|\cdot\|_X$ corresponds to the unit ball $|x|<1$). For the sake of completeness we repeat the proof of this result in Lemma~\ref{equ} below. Obviously, norm (\ref{normX}) generates the stronger topology than one of $L^1_{loc}(\R^n)$.

We denote by $F$ the closed set of points $u\in\R$ such that the flux components $\varphi(u)\cdot\xi$ are not affine on any vicinity of $u$ where the diffusion coefficients $a(u)\xi\cdot\xi=0$ (almost everywhere in this vicinity) for all $\xi\in G'$, $\xi\not=0$. In the case when $G=H=\R^n$ (and nonzero vectors $\xi\in G'$ do not exist), we define $F$ as the set of $u$ such that there is no such vicinity of $u$ where the entire vector $\varphi(u)$ is affine while the entire diffusion matrix $a(u)=0$ (a.e. in this vicinity).
Our main result is the following decay property.

\begin{theorem}\label{thM}
Assume that the following {\bf nonlinearity-diffusivity} condition is satisfied: for all $a<m$, $b>m$ the intervals $(a,m)$, $(m,b)$ intersect with $F$: $(a,m)\cap F\not=\emptyset$, $(m,b)\cap F\not=\emptyset$. Suppose that $u(t,x)$ is an e.s. of (\ref{1}), (\ref{2}). Then
\begin{equation}\label{dec}
\esslim_{t\to+\infty}\|u(t,\cdot)-m\|_X=0.
\end{equation}
\end{theorem}

The condition $H=\R^n$ means that the function $p(x)$ is constant, $p\equiv m$. In this case, the requirement of Theorem~\ref{thM} reduces to the condition that on any semivicinity
$(a,m)$, $(m,b)$ of the mean $m$ either the flux vector $\varphi(u)$ is not affine or the diffusion matrix $a(u)\not\equiv 0$. When $p\equiv m=0$, Theorem~\ref{thM} was proved in \cite[Theorem~1.4]{PaJDE2} (and in \cite{PaSIMA2} for conservation laws (\ref{con})). The case of arbitrary $m$
reduces to the case $m=0$ by the change $u\to u-m$, $\varphi(u)\to\varphi(u+m)$, $a(u)=a(u+m)$. Thus, we may suppose in the sequel that
$p\not\equiv\const$ and therefore the lattice $G'$ is not trivial.

Remark that the nonlinearity-diffusivity requirement in Theorem~\ref{thM} implies that $m\in F$ because of closeness of this set. Generally, under this weaker condition $m\in F$ the decay property fails, cf. section~\ref{ex} below. But,
in periodic case $v\equiv 0$, the decay property (\ref{dec}) holds under the weaker condition $m\in F$,
that is,
\begin{align}\label{gn}
\forall\xi\in G', \xi\not=0 \mbox{ either the flux components } \varphi(u)\cdot\xi \nonumber \mbox{ are not affine or} \\  \mbox{the diffusion coefficients } a(u)\xi\cdot\xi\not\equiv 0 \mbox{ on any vicinity of } m.
\end{align}
In the standard case when $G$ is a lattice (that is, when $\dim H=0$) it was proved in \cite[Theorem~1.1]{PaJDE1}, see also earlier paper \cite{ChPer2}, where the decay property was established under a more restrictive nonlinearity-diffusivity assumption. The general case of arbitrary $H$ easily reduces to the case $\dim H=0$. We provide the details in the following theorem.

\begin{theorem}\label{thPD}
Suppose that the initial function $u_0=p(x)$ is periodic with a group of periods $G$, and condition (\ref{gn}) is satisfied. Then the e.s. $u=u(t,x)$ of problem (\ref{1}), (\ref{2}) exhibits the decay property
\begin{equation}\label{decp}
\esslim_{t\to+\infty} u(t,\cdot)=m=\int_{\T^d} u_0(x)dx \ \mbox{ in } L^1(\T^d).
\end{equation}
\end{theorem}

\begin{proof}
Observe that for all $e\in G$ $u_0(x+e)=u_0(x)$ a.e. in $\R^n$. Obviously, $u(t,x+e)$ is an e.s. of (\ref{1}), (\ref{2})   with the same initial data $u_0(x)$. By the uniqueness of an e.s., known in the case of periodic initial function, we claim that $u(t,x+e)=u(t,x)$ a.e. in $\Pi$, that is, $u(t,x)$ is $G$-periodic in the space variables.
In particular, $u(t,\cdot)\in L^1(\T^d)$ for a.e. $t>0$ and relation (\ref{decp}) is well-defined.
We choose a non-degenerate linear operator $Q$ in $\R^n$, which transfers the space $H$ into the standard subspace
$$
\R^{n-d}=\{ \ x=(x_1,\ldots,x_n) \ | \ x_i=0 \ \forall i=1,\ldots,d \ \}.
$$
After the change $y=Qx$ our problem reduces to the problem
\begin{equation}\label{1r}
v_t+\div_y(\tilde\varphi(v)-\tilde a(v)\nabla_y v)=0, \quad v(0,y)=v_0(y)\doteq u_0(Q^{-1}(y)),
\end{equation}
where $\tilde\varphi(v)=Q\varphi(v)$, $\tilde a(v)=Qa(v)Q^*$, where $Q^*=Q^\top$ is a conjugate operator. As is easy to verify, $u(t,x)=v(t,Qx)$, where $v(t,y)$ is an e.s. of (\ref{1r}). Observe that $v_0(y)$ is periodic with the group of periods
$\tilde G=Q(G)$. Therefore, the e.s. $v(t,y)$  is space periodic with the group of periods containing $\tilde G$. In particular, the functions $v_0(y)$, $v(t,y)$ are constant in directions $\R^{n-d}=Q(H)$: $v_0(y)=v_0(y_1,\ldots,y_d)$, $v(t,y)=v(t,y_1,\ldots,y_d)$ with $v_0(y')\in L^\infty(\R^d)$, $v(t,y')\in L^\infty(\R_+\times\R^d)$.  This readily implies that $v(t,y')$ is an e.s. of the low-dimensional problem
\begin{align}\label{1l}
v_t+\div_{y'}(\tilde\varphi(v)-\tilde a(v)\nabla_{y'} v)=\nonumber\\ v_t+\sum_{i=1}^d \left(\frac{\partial}{\partial y_i}\tilde\varphi_i(u)-\frac{\partial}{\partial y_i}\sum_{j=1}^d (\tilde a)_{ij}(v)v_{y_j}\right)=0, \quad v(0,y')=v_0(y').
\end{align}
Observe that $v_0(y')$ is periodic with the lattice of periods
$$
\tilde L=\tilde G\cap\R^d=\{ \ y=(y_1,\ldots,y_d)\in\R^d \ | \ (y_1,\dots,y_d,0,\ldots,0)\in \tilde G \ \}
$$
(by Lemma~\ref{lem2} it is indeed a lattice). Using again Lemma~\ref{lem2}, we find that the dual lattice
$\tilde L'=\tilde G'=(Q^*)^{-1} G'\subset\R^d$. Observe that for each nonzero
$\zeta\in\tilde L'$ the vector $\xi=Q^*\zeta\in G'$, and
\begin{align*}
\zeta\cdot\tilde\varphi(v)=\zeta\cdot Q\varphi(v)=Q^*\zeta\cdot\varphi(v)=\xi\cdot\varphi(v), \\
\tilde a(v)\zeta\cdot\zeta= Qa(v)Q^*\zeta\cdot\zeta=a(v)Q^*\zeta\cdot Q^*\zeta=a(v)\xi\cdot\xi
\end{align*}
By condition (\ref{gn}) we claim that in any vicinity of $m$ for every $\zeta\in\tilde L'$ either the function $v\to\zeta\cdot\tilde\varphi(v)$ is not affine or the function $v\to\tilde a(v)\zeta\cdot\zeta$ is not zero on a set of positive measure. Observe that $\zeta\in\tilde L'\subset\R^d$, and we may replace $\tilde\varphi(v)$, $\tilde a(v)$ by $P\tilde\varphi(v)$, $P\tilde a(v)P^*$, respectively, where $P:\R^n\to\R^d$ is the orthogonal projection. This means that  equation (\ref{1l}) satisfies the nonlinearity-diffusivity condition subject to the lattice $\tilde L$ in $\R^d$. By the decay property \cite[Theorem~1.1]{PaJDE1} applied to the e.s. $v(t,y')$ of (\ref{1l}) we claim that
\begin{equation}\label{decp1}
\esslim_{t\to+\infty} v(t,\cdot)=\tilde m=\int_{\tilde\T^d} v_0(y')dy' \ \mbox{ in } L^1(\tilde\T^d),
\end{equation}
where $\tilde\T^d=\R^d/\tilde L=\R^n/\tilde G$ is a torus corresponding to the lattice $\tilde L$, and $dy'$ denotes the normalized Lebesgue measure on this torus. Making the change of variables $y=Qx$, which induces an isomorphism $Q:\T^d\to\tilde\T^d$, we find that
$$
\tilde m=\int_{\T^d} u_0(x)dx=m,
$$
and relation (\ref{decp}) follows from (\ref{decp1}). The proof is complete.
\end{proof}
Remark also that Theorem~\ref{thPD} follows from result \cite{PaLJM} on decay of almost periodic (in Besicovitch sense) e.s. In \cite{PaLJM} the decay property (in the Besicovitch space) was established under the same assumption as in Theorem~\ref{thPD} but with the set $G'$ replaced by the additive group $M(u_0)$ generated by the spectrum of initial function. In the case when $u_0$ is periodic, its Fourier expansion have the form $u_0(x)=\sum\limits_{\lambda\in G'}a_\lambda e^{2\pi\lambda\cdot x}$ and, in particular, the spectrum $Sp(u_0)=\{\lambda\in\R^n | a_\lambda\not=0\}$ as well as the group $M(u_0)$ are contained in the lattice $G'$. This implies that $M(u_0)$ is a lattice too. By the duality, in order to prove the equality $M(u_0)=G'$ it is sufficient to verify the inclusion $M(u_0)'\subset G$. If $e\in M(u_0)'$ then
$\lambda\cdot e\in\Z$ for each $\lambda\in Sp(u_0)$. This implies that the Fourier series of $u_0(x+e)$ is the same as of $u_0(x)$. Hence, $u_0(x+e)=u_0(x)$ a.e. in $\R^n$ and $e\in G$. We see that $M(u_0)'\subset G$ and, therefore, $M(u_0)=G'$. Thus, the assumption of \cite{PaLJM} reduces to our assumption. Taking also into account that in periodic case the Besicivitch norm is equivalent to the norm of $L^1(\T^d)$, we conclude that the decay property stated in Theorem~\ref{thPD} holds.

Let us show that condition (\ref{gn}) is precise, that is, if this condition fails, then there exists a periodic function exactly with the group of periods $G$ such that the corresponding e.s. does not satisfy the decay property. In fact, if (\ref{gn}) fails, we can find a nonzero vector $\xi\in G'$ and constants $\delta>0$, $k\in\R$ such that $\xi\cdot\varphi(u)-ku=\const$ on the segment $|u-m|\le\delta$ while $a(u)\xi\cdot\xi=0$ a.e. on this segment. Since the matrix $a(u)\ge 0$, the equality $a(u)\xi\cdot\xi=0$ holds if and only if $a(u)\xi=0$. We define the hyperspace
$E=\{x\in\R^n \ | \ \xi\cdot x=0\}$. The linear functional $\xi$ is a homomorphism of the group $G$ into $\Z$. The range
of this homomorphism is a subgroup $r\Z\subset\Z$ for some $r\in\N$. Denote by $G_1=E\cap G$ the kernel of $\xi$.
There exists an element $e_0\in G$ such that $\xi\cdot e_0=r$. Then, as is easy to verify, the map $(e,m)\to e+me_0$
forms a group isomorphism of $G_1\oplus\Z$ onto $G$. We choose $n-1$ independent vectors $\zeta_i$ such that $\zeta_i\cdot e_0=0$, $i=1,\ldots,n-1$, and complement them to a basis attaching the vector $\zeta_n=\xi$. We make the linear change $y=y(t,x)$
\begin{equation}\label{ch}
y_i=\zeta_i\cdot x, \ i=1,\ldots,n-1, \quad y_n=\zeta_n\cdot x-kt,
\end{equation}
which reduces (\ref{1}) to the equation
\begin{equation}\label{ch1}
u_t+\div_y(\tilde\varphi(u)-\tilde a(u)\nabla_y u)=0
\end{equation}
such that $\tilde\varphi_n(u)=\xi\varphi(u)-ku=\const$, $\tilde a_{in}=\tilde a_{ni}=a(u)\xi\cdot\zeta_i=0$, $i=1,\ldots,n$, a.e. on the segment $|u-m|\le\delta$.
If an initial data $\tilde u_0(y)$ satisfies the condition $|\tilde u_0-m|\le\delta$ then a corresponding e.s. $u=\tilde u(t,x)$ of (\ref{ch1}) also satisfies the condition $|\tilde u(t,x)-m|\le\delta$ a.e. on $\Pi$, by the maximum-minimum principle \cite[Corollary~2.2]{PaJDE2}. Since $\tilde\varphi_n(u)$ is constant on the segment $|u-m|\le\delta$, while the diffusion coefficients $\tilde a_{ij}(u)$ with $\max(i,j)=n$ vanish a.e. on this segment, then $\tilde u(t,y)$ is an e.s. of the equation
\begin{equation}\label{ch2}
u_t+\div_{y'}(\tilde\varphi(u)-\tilde a(u)\nabla_{y'}u)=u_t+\sum_{i=1}^{n-1}(\tilde\varphi_i(u))_{y_i}-\bigl(\sum_{i,j=1}^{n-1}\tilde a_{ij}(u)u_{y_j}\bigr)_{y_i}=0,
\end{equation}
where $y'=(y_1,\ldots,y_{n-1})\in\R^{n-1}$. This readily implies that for a.e. fixed $y_n\in\R$ the function
$\tilde u(t,y',y_n)$ is an e.s. of the Cauchy problem for the low-dimensional equation (\ref{ch2}), considered in
the domain $\R_+\times\R^{n-1}$, with the corresponding initial function $\tilde u_0(y',y_n)$. Assume that the function
$\tilde u_0(y',y_n)$ is $y'$-periodic (with some group of periods) with the mean value $m(y_n)=\frac{1}{|P|}\int_{P} u_0(y',y_n)dy'$, $P\subset\R^{n-1}$ being the periodicity cell (or, the same, the corresponding torus). Then for a.e. $y_n\in\R$ the mean value of $\tilde u(t,\cdot,y_n)$ does not depend on $t$ and equals $m(y_n)$ (see, for instance, \cite[Corollary~2.1]{PaJDE1}). If this function $m(y_n)$ is not constant (a.e. in $\R$) then the e.s. $\tilde u(t,y)$ cannot satisfy the decay property. In fact, if $\tilde u(t,\cdot)-m\to 0$ as $t\to+\infty$ in $L^1_{loc}(\R^n)$, then for each interval $I\subset\R$
$$
\left|\int_I (m(y_n)-m)dy_n\right|=\frac{1}{|P|}\left|\int_{P\times I}(\tilde u(t,y)-m)dy\right|\le\frac{1}{|P|}\int_{P\times I}|\tilde u(t,y)-m|dy,
$$
which implies, in the limit as $t\to+\infty$, that $\int_I (m(y_n)-m)dy_n=0$. Since $I$ is an arbitrary interval, we find that $m(y_n)=m$ a.e. in $\R$, which contradicts to our assumption.

Now we choose a function $v(x)\in C(E)$ such that $\|v\|_\infty\le\delta/2$ and that $v(x)$ is periodic with the group of periods $G_1$ and with zero mean value. Since $G_1=H\oplus (E\cap L_0)$, such $v(x)$ actually exists, this function is constant in the direction $H$ and is periodic in $E\cap H^\perp$ with exactly the lattice of periods $E\cap L_0$. We set
$$
u_0(x)=m+v(\pr(x))+\frac{\delta}{2}\sin(2\pi\xi\cdot x/r),
$$
where $\pr(x)\in E$ is the projection of $x$ on $E$ along the vector $e_0$ (so that $x-\pr(x)\parallel e_0$).
Then $\|u_0-m\|_\infty\le\delta$. Let us show that $u_0(x)$ is periodic with the group of periods $G$. Since vectors
$e\in G_1$ and $e_0$ are periods of $u_0$, then the group $G=G_1+\Z e_0$ consists of period of $u_0$. On the other hand, if
$e\in\R^n$ is a period of $u_0$ then it can be decomposed into a sum $e=e_1+\lambda e_0$, where $e_1\in E$, $\lambda\in\R$.
For $x=x'+se_0$, $x'\in E$, we have
$$
u_0(x+e)=m+v(x'+e_1)+\frac{\delta}{2}\sin(2\pi(s+\lambda))=u_0(x)=m+v(x')+\frac{\delta}{2}\sin(2\pi s).
$$
Averaging this equality over $x'$, we obtain that $\sin(2\pi(s+\lambda))=\sin(2\pi s)$ for all $s\in\R$, which implies that $\lambda\in\Z$ and that $v(x'+e_1)=v(x')$ for all $x'\in E$. Therefore, $e_1\in G_1$
(remind that $G_1$ is the group of periods of $v$). Hence $e=e_1+\lambda e_0\in G_1+\Z e_0=G$. We proved that the group
of periods of $u_0$ is exactly $G$. It is clear that $m$ is the mean value of $u_0$.

Now, we are going to show that an e.s. $u(t,x)$ of (\ref{1}), (\ref{2}) with the chosen initial data does not satisfy decay property (\ref{decp}). After the change (\ref{ch}) the initial function $u_0(x)$ transforms into $\displaystyle \tilde u_0(y)=m+\tilde v(y')+\frac{\delta}{2}\sin(2\pi y_n/r)$,
the function $\tilde v(y')\in C(\R^{n-1})$ is determined by the identity $v(x)=\tilde v(y(x))$, where $y(x)=y(0,x)$, $x\in E$, is a linear isomorphism $E\to\R^{n-1}$. Obviously, the function $\tilde v(y')$ is periodic with the group of periods $y(G_1)$ and zero mean value. Therefore, the mean value of initial data over the variables $y'$ equals $\displaystyle m(y_n)=m+\frac{\delta}{2}\sin(2\pi y_n/r)$ and it is not constant. In this case it has been already demonstrated that an e.s. $\tilde u(t,y)$ of the Cauchy problem
for equation (\ref{ch1}) does not satisfy the decay property. Due to the identity $u(t,x)=\tilde u(t,y(t,x))$, we see that an e.s. of original problem does not satisfy (\ref{decp}) either.

\section{Proof of the main results}
\subsection{Auxiliary lemmas}

\begin{lemma}\label{lem2}
Let $G$ be the group of periods of a periodic function $p(x)\in L^\infty(\R^n)$, and let, as in Introduction, $H$ be a maximal linear subspace of $G$, $L_0=G\cap H^\perp$, and let $G'$ be a dual group to $G$. Then

(i) $L_0$ is a lattice of dimension $d=\dim H^\perp$;

(ii) $G'$ is a lattice in $H^\perp$ of dimension $d$, and $G'=L_0'$ in $H^\perp$.
\end{lemma}

\begin{proof}
(i) We have to prove that the group $L_0$ is discrete, that is, all its points are isolated. Since $L_0$ is a group it is sufficient to show that $0$ is an isolated point of $L_0$. Assuming the contrary, we find a sequence $h_k\in L_0$, such that $h_k\not=0$, $h_k\to 0$ as $k\to\infty$. By compactness of the unit sphere $|x|=1$, we may suppose that the sequence
$|h_k|^{-1}h_k\to\xi$ as $k\to\infty$, where $\xi\in H^\perp$, $|\xi|=1$. Let $w(x)\in C_0^1(\R^n)$ and
$v(x)$ be the convolution $v=p\ast w(x)=\int_{\R^n}p(x-y)w(y)dy$. By known property of convolution $v(x)\in C^1(\R^n)$. Further, for each $e\in G$ and $x\in\R^n$
$$
v(x+e)=\int_{\R^n}p(x-y+e)w(y)dy=\int_{\R^n}p(x-y)w(y)dy=v(x)
$$
and in particular $v(x+h_k)=v(x)$ $\forall k\in\N$. Since $v$ is differentiable,
\begin{equation}\label{lm1}
0=|h_k|^{-1}(v(x+h_k)-v(x))=\nabla v(x)\cdot |h_k|^{-1}h_k+\varepsilon_k,
\end{equation}
where $\varepsilon_k\to 0$ as $k\to\infty$. Passing in (\ref{lm1}) to the limit as $k\to\infty$, we obtain that
$\frac{\partial v(x)}{\partial\xi}=\nabla v(x)\cdot\xi=0$ for all $x\in\R^n$. Therefore, $v$ is constant in the direction $\xi$: $v(x+s\xi)=v(x)$ for all $s\in\R$, $x\in\R^n$. We choose a nonnegative function $w(x)\in C_0^1(\R^n)$ such that $\int_{\R^n} w(x)dx=1$ and set $\omega_r(x)=r^n\omega(rx)$, $r\in\N$. This sequence (an approximate unity) converges to Dirac $\delta$-function as $r\to\infty$ weakly in the space of distributions $\D'(\R^n)$. The corresponding sequence of averaged functions $v_r=p\ast w_r(x)$ converges to $p$ as $r\to\infty$ in $L^1_{loc}(\R^n)$. As was already established, the functions $v_r$ are constant in the direction $\xi$. Therefore, for each $\alpha\in\R$, $R>0$
$$
\int_{|x|<R} |v_r(x+\alpha\xi)-v_r(x)|dx=0.
$$
Passing in this relation to the limit as $r\to\infty$, we obtain that
$$
\int_{|x|<R} |p(x+\alpha\xi)-p(x)|dx=0 \quad \forall R>0,
$$
which implies that $p(x+\alpha\xi)=p(x)$ a.e. in $\R^n$, that is, $\alpha\xi\in G$. Thus,
the linear subspace $H_1=\{ \ x+\alpha\xi \ | \ x\in H, \alpha\in\R \ \}\subset G$. Since $\xi\in H^\perp$, $\xi\not=0$,
then $H\subsetneq H_1$. But this contradicts to the maximality of $H$. This contradiction proves that $L_0$ is a discrete additive subgroup of $\R^n$, i.e., a lattice. By the construction, $G=H\oplus L_0$. Since $G$ generates the entire space $\R^n$, then $L_0$ must generate $H^\perp$, that is, $\dim L_0=d$.

(ii) Since $\dim G=n$, we can choose a basis $e_k$, $k=1,\ldots,n$, of the linear space $\R^n$ lying in $G$. We define
$R=\max\limits_{k=1,\ldots,n}|e_k|$, $\delta=1/R$. Let $\xi\in G'$, $|\xi|<\delta$. Then
$|\xi\cdot e_k|\le|\xi||e_k|\le |\xi|R<1$. Since $\xi\cdot e_k\in\Z$ we claim that $\xi\cdot e_k=0$ for all $k=1,\ldots,n$.
Since $e_k$, $k=1,\ldots,n$, is a basis, this implies that $\xi=0$. We obtain that the ball $|\xi|<\delta$ contains only zero element of the group $G'$. This means that this group is discrete and therefore it is a lattice. If $\xi\in G'$, $e\in H$ then $\alpha\xi\cdot e=\xi\cdot\alpha e\in\Z$ for all $\alpha\in\R$. This is possible only if $\xi\cdot e=0$. This holds for every $e\in H$, that is, $\xi\in H^\perp$. Obviously, for such $\xi$, the requirement $\xi\in G'$ reduces to the condition $\xi\cdot e\in\Z$ for all $e\in L_0$. We conclude that $G'=L_0'$. Since $L_0$ is a lattice, this in particular implies that $\dim G'=\dim L_0=d$.
\end{proof}

\begin{lemma}\label{equ}
The norms $\|\cdot\|_V$ defined in (\ref{normV}) are mutually equivalent.
\end{lemma}

\begin{proof}
Let $V_1,V_2$ be open bounded sets in $\R^n$, and $K_1=\Cl V_1$ be the closure of $V_1$. Then $K_1$ is a compact set while $y+V_2$, $y\in K_1$, is its open covering.
By the compactness there is a finite set $y_i$, $i=1,\ldots,m$, such that
$\displaystyle K_1\subset \bigcup\limits_{i=1}^m (y_i+V_2)$. This implies that for every $y\in\R^n$ and $u=u(x)\in L^\infty(\R^n)$
$$
\int_{y+V_1}|u(x)|dx\le\sum_{i=1}^m \int_{y+y_i+V_2}|u(x)|dx\le m\|u\|_{V_2}.
$$
Hence, $\forall u=u(x)\in L^\infty(\R^n)$
$$
\|u\|_{V_1}=\sup_{y\in\R^n}\int_{y+V_1}|u(x)|dx\le m\|u\|_{V_2}.
$$
Changing the places of $V_1$, $V_2$, we obtain the inverse inequality
$\|u\|_{V_2}\le l\|u\|_{V_1}$ for all $u\in L^\infty(\R^n)$, where $l$ is some positive constant. This completes the proof.
\end{proof}

\begin{proposition}\label{pro1}
Let $G$ be a lattice and values $\alpha^+,\alpha^-\in F$ be such that $\alpha^-<m<\alpha^+$. Then an e.s. $u(t,x)$ of (\ref{1}), (\ref{2})
satisfies the property
$$
\esslimsup_{t\to +\infty}\|u(t,\cdot)-m\|_X\le 2^n(\alpha^+-\alpha^-).
$$
\end{proposition}
\begin{proof}
Let $e_k$, $k=1,\ldots,n$, be a basis of the lattice $G$. We define for $r\in\N$ the parallelepiped
$$
P_r=\left\{ \ x=\sum_{k=1}^n x_ke_k: \ -r/2\le x_k<r/2, k=1,\ldots,n \ \right\}.
$$
It is clear that $P_r$ is a fundamental parallelepiped for a lattice $rG\subset G$.
We introduce the functions
$$
v_r^+(x)=\sup_{e\in G} v(x+re), \quad v_r^-(x)=\inf_{e\in G} v(x+re), \quad V_r(x)=\sup_{e\in G} |v(x+re)|.
$$
Since $G$ is countable, these functions are well-defined in $L^\infty(\R^n)$,
and $|v_r^\pm|\le V_r(x)\le C_0\doteq\|v\|_\infty$ for a.e. $x\in\R^n$. It is clear that $v_r^\pm(x)$ are $rG$-periodic and
\begin{equation}\label{p2v}
v_r^-(x)\le v(x)\le v_r^+(x).
\end{equation}
Let us show that under condition (\ref{van})
\begin{equation}\label{l3}
M_r=\frac{1}{|P_r|}\int_{P_r} V_r(x)dx\to 0 \ \mbox{ as } r\to+\infty.
\end{equation}
For that we fix $\varepsilon>0$ and define the set $A=\{ \ x\in\R^n: \ |v(x)|>\varepsilon \ \}$.
In view of (\ref{van}) the measure of this set is finite, $\meas A=q<+\infty$. We also define the sets
$$
A_r^e=\{ \ x\in P_r: \ x+re\in A \ \}\subset P_r, \quad r>0, \ e\in G, \quad A_r=\bigcup_{e\in G} A_r^e.
$$
By the translation invariance of Lebesgue measure and the fact that $\R^n$ is the disjoint union of the sets $re+P_r$, $e\in G$, we have
$$
\sum_{e\in G}\meas A_r^e=\sum_{e\in G}\meas (re+A_r^e)=\sum_{e\in G}\meas (A\cap (re+P_r))=\meas A=q.
$$
This implies that
\begin{equation}\label{Ar}
\meas A_r\le\sum_{e\in G}\meas A_r^e=q.
\end{equation}
If $x\notin A_r$ then $|v(x+re)|\le\varepsilon$ for all $e\in G$, which implies that $V_r(x)\le\varepsilon$. Taking (\ref{Ar}) into account, we find
$$
\int_{P_r} V_r(x)dx=\int_{A_r} V_r(x)dx+\int_{P_r\setminus A_r} V_r(x)dx\le C_0\meas A_r+\varepsilon\meas P_r\le C_0q+\varepsilon|P_r|.
$$
It follows from this estimate that
$$\limsup_{r\to+\infty} M_r\le\lim_{r\to+\infty}\left(\frac{C_0q}{|P_r|}+\varepsilon\right)=\varepsilon$$ and since $\varepsilon>0$ is arbitrary, we conclude that (\ref{l3}) holds.
Let
$$
\varepsilon_r^\pm=\frac{1}{|P_r|}\int_{P_r} v_r^\pm(x)dx
$$
be mean values of $rG$-periodic functions $v_r^\pm(x)$.
In view of (\ref{l3})
\begin{equation}\label{l4}
|\varepsilon_r^\pm|\le M_r\mathop{\to}_{r\to\infty} 0.
\end{equation}
By (\ref{l4}) we claim that $|\varepsilon_r^\pm|<\min(\alpha^+-m,m-\alpha^-)$ for sufficiently large $r\in\N$.
We introduce for such $r$ the $rG$-periodic functions
$$
u_0^+(x)=p(x)+v_r^+(x)+\alpha^+-m-\varepsilon_r^+, \quad u_0^-(x)=p(x)+v_r^-(x)-(m-\alpha^-+\varepsilon_r^-)
$$
with the mean values $\alpha^+,\alpha^-$, respectively. In view of (\ref{p2v}) and the conditions
$\alpha^+-m-\varepsilon_r^+>0$, $m-\alpha^-+\varepsilon_r^->0$, we have
\begin{equation}\label{p2}
u_0^-(x)\le u_0(x)\le u_0^+(x).
\end{equation}
Let $u^\pm$ be unique (by \cite[Theorem~1.3]{PaJDE2}) e.s. of (\ref{1}), (\ref{2}) with initial functions
$u_0^\pm$, respectively. Taking into account that $(rG)'=\frac{1}{r}G'$, we see that condition (\ref{gn}), corresponding to the lattice $rG$ and the mean values $\alpha^-,\alpha^+$, is satisfied. By Theorem~\ref{thPD} (or \cite[Theorem~1.1]{PaJDE1}) we find that
\begin{equation}\label{l5}
\esslim_{t\to+\infty}\int_{P_r}|u^\pm(t,x)-\alpha^\pm|dx=0.
\end{equation}
By the periodicity, for each $y\in\R^n$
$$
\int_{y+P_r}|u^\pm(t,x)-\alpha^\pm|dx=\int_{P_r}|u^\pm(t,x)-\alpha^\pm|dx,
$$
which readily implies that for $V=\Int P_r$
$$
\|u^\pm(t,x)-\alpha^\pm\|_V=\int_{P_r}|u^\pm(t,x)-\alpha^\pm|dx.
$$
In view of Lemma~\ref{equ} we have the estimate
$$\|u^\pm(t,x)-\alpha^\pm\|_X\le C\int_{P_r}|u^\pm(t,x)-\alpha^\pm|dx, \ C=C_r=\const.$$
By (\ref{l5}) we claim that
\begin{equation}\label{l5a}
\esslim_{t\to+\infty}\|u^\pm(t,\cdot)-\alpha^\pm\|_X=0.
\end{equation}
Let $u=u(t,x)$ be an e.s. of the original problem (\ref{1}), (\ref{2}) with initial data $u_0(x)$. Since the functions $u_0^\pm$ are periodic, then it follows from (\ref{p2}) and the comparison principle \cite[Theorem~1.3]{PaJDE2} that $u^-\le u\le u^+$ a.e. in $\Pi$. This readily implies the relation
\begin{align}\label{l6}
\|u(t,\cdot)-m\|_X\le \|u^-(t,\cdot)-m\|_X+\|u^+(t,\cdot)-m\|_X\le \nonumber\\
\|u^-(t,x)-\alpha^-\|_X+\|u^+(t,x)-\alpha^+\|_X+c(\alpha^+-m+m-\alpha^-),
\end{align}
where $c\le 2^n$ is Lebesgue measure of the unit ball $|x|<1$ in $\R^n$.
In view of (\ref{l5a}) it follows from (\ref{l6}) in the limit as $t\to+\infty$
that
$$
\esslimsup_{t\to+\infty}\|u(t,\cdot)-m\|_X\le c(\alpha^+-\alpha^-)\le 2^n(\alpha^+-\alpha^-),
$$
as was to be proved.
\end{proof}

\subsection{Proof of Theorem~\ref{thM}}

We are going to establish that the statement of Proposition~\ref{pro1} remains valid in the case of arbitrary $G$. We will suppose that $\dim H<n$, otherwise $p\equiv\const$ and this case has been already considered in Introduction.

\begin{proposition}\label{pro2} Assume that $m\in(\alpha^-,\alpha^+)$, where $\alpha^\pm\in F$, and $u=u(t,x)$ is an e.s. of (\ref{1}), (\ref{2}). Then,
\begin{equation}\label{mest}
\esslimsup_{t\to+\infty}\|u(t,\cdot)-m\|_X\le 2^{n+3}(\alpha^+-\alpha^-).
\end{equation}
\end{proposition}
\begin{proof}
Suppose firstly that $v(x)\in L^\infty(\R^n)$ is a finite function, that is, its closed support $\supp v$ is compact.
Let $A=H+\supp v=H\oplus K$, where $K$ is the orthogonal projection of $\supp v$ on the space $H^\perp$.
We define the functions $v^\pm(x)=\pm\|v\|_\infty\chi_A(x)$, where $\chi_A(x)$ is the indicator function of the set $A$.
We define functions $u_0^\pm(x)=p(x)+v^\pm(x)$ and let $u^-(t,x)$ be the smallest e.s. of (\ref{1}), (\ref{2}) with initial data $u_0^-(x)$, $u^+(t,x)$ be the largest e.s. of (\ref{1}), (\ref{2}) with initial data $u_0^+(x)$, existence of such e.s. was established in \cite[Theorems~1.1]{PaJDE2}. Since $u_0^-(x)\le u_0(x)\le u_0^+(x)$ a.e. on $\R^n$, we derive that
$u^-\le u\le u^+$ by the property of monotone dependence of the smallest and the largest e.s. on initial data, cf.
\cite[Theorem~1.2]{PaJDE2}. Observe that the initial functions $u_0^\pm(x)$ are constant in direction $H$. Therefore, for any $e\in H$, the functions $u^\pm(t,x+e)$ are the largest and the smallest e.s. of the same problems as $u^\pm(t,x)$. By the uniqueness we claim that $u^\pm(t,x+e)=u^\pm(t,x)$ a.e. in $\Pi$. Hence, $u^\pm(t,x)=u^\pm(t,x')$, $x'=\pr_{H^\perp} x$. As is easy to see, $u^\pm(t,x')$ are the largest and the smallest e.s. of $d$-dimensional problem
$$
u_t+\div_{x'}(\tilde\varphi(u)-\tilde a(u)\nabla_{x'}u)=0, \quad u(0,x')=u_0^\pm(x')
$$
on the subspace $H^\perp$, where $\tilde\varphi(u)=P\varphi(u)$, $\tilde a(u)=Pa(u)P^*$, $P=\pr_{H^\perp}$ being the orthogonal projection on the space $H^\perp$. In the same way as in the proof of Theorem~\ref{thPD} this problem can be written in the standard way (like (\ref{1r})~) by an appropriate change of the space variables. Since the initial functions $u_0^\pm(x')=p(x')+v^\pm(x')$, where $p(x')$ is periodic with the lattice of periods $L_0$, while $v^\pm(x')$ are bounded functions with compact support $K$ (so that $v^\pm(x')\in L_0^\infty(H^\perp)$), we may apply Proposition~\ref{pro1}. By this proposition,
\begin{align}\label{3}
\esslimsup_{t\to +\infty} \|u^\pm(t,x)-m\|_X\le\esslimsup_{t\to +\infty} 2^{n-d}\|u^\pm(t,x')-m\|_X\le\nonumber\\  2^{n-d}2^d(\alpha^+-\alpha^-)= 2^n(\alpha^+-\alpha^-),
\end{align}
where the first $X$-norm is taken in $L^\infty(\R^n)$ while the second $X$-norm is in $L^\infty(H^\perp)$.
Since the e.s. $u(t,x)$ is situated between $u^-$ and $u^+$, then $$\|u(t,\cdot)-m\|_X\le\|u^+(t,\cdot)-m\|_X+\|u^-(t,\cdot)-m\|_X$$ and in view of (\ref{3})
\begin{equation}\label{4}
\esslimsup_{t\to +\infty} \|u(t,\cdot)-m\|_X\le 2^{n+1}(\alpha^+-\alpha^-).
\end{equation}
Now we suppose that $v\in L^1(\R^n)\cap L^\infty(\R^n)$. For fixed $\varepsilon>0$ we can find a function $\tilde v\in L^\infty(\R^n)$ with compact support such that $\|v-\tilde v\|_1\le\varepsilon$. We denote by $u^+=u^+(t,x)$, $u^-=u^-(t,x)$ the largest and the smallest e.s. of (\ref{1}), (\ref{2}) with initial function $u_0=p(x)+v(x)$. Similarly,
by $\tilde u^+=\tilde u^+(t,x)$, $\tilde u^-=\tilde u^-(t,x)$ we denote the largest and the smallest e.s. of (\ref{1}), (\ref{2}) with initial function $\tilde u_0=p(x)+\tilde v(x)$. It is known, cf. \cite[Theorems~1.2]{PaJDE2}, that the largest and the smallest e.s. exhibit the $L^1$-contraction property. In particular, for a.e. $t>0$
\begin{equation}\label{5}
\int_{\R^n}|u^\pm(t,x)-\tilde u^\pm(t,x)|dx\le\int_{\R^n}|u_0(x)-\tilde u_0(x)|dx=\|v-\tilde v\|_1<\varepsilon.
\end{equation}
Since the function $\tilde v$ has finite support, relation (\ref{4}) holds for the e.s. $\tilde u^\pm(t,x)$, i.e.,
\begin{equation}\label{6}
\esslimsup_{t\to +\infty} \|\tilde u^\pm(t,\cdot)-m\|_X\le 2^{n+1}(\alpha^+-\alpha^-).
\end{equation}
In view of (\ref{5})
$$\|u^\pm(t,\cdot)-\tilde u^\pm(t,\cdot)\|_X\le \|u^\pm(t,\cdot)-\tilde u^\pm(t,\cdot)\|_1<\varepsilon$$
and (\ref{6}) implies the estimates
$$
\esslimsup_{t\to +\infty} \|u^\pm(t,\cdot)-m\|_X\le 2^{n+1}(\alpha^+-\alpha^-)+\varepsilon,
$$
and since $\varepsilon>0$ is arbitrary, we find that
\begin{equation}\label{7}
\esslimsup_{t\to +\infty} \|u^\pm(t,\cdot)-m\|_X\le 2^{n+1}(\alpha^+-\alpha^-).
\end{equation}
Since $u^-\le u\le u^+$, then $\|u(t,\cdot)-m\|_X\le \|u^+(t,\cdot)-m\|_X+\|u^-(t,\cdot)-m\|_X$ and it follows from (\ref{7}) that
\begin{equation}\label{8}
\esslimsup_{t\to +\infty} \|u(t,\cdot)-m\|_X\le 2^{n+2}(\alpha^+-\alpha^-).
\end{equation}
In the general case $v\in L_0^\infty(\R^n)$ we choose such $\delta>0$ that $\alpha^-<m-\delta<m+\delta<\alpha^+$ and set
$v_+(x)=\max(v(x)-\delta,0)$, $v_-(x)=\min(v(x)+\delta,0)$. Observe that these functions vanish outside of the set $|v(x)|>\delta$ of finite measure. Therefore, $v_\pm\in L^1(\R^n)$. Obviously,
\begin{equation}\label{9}
u_0(x)\le u_{0+}(x)=p(x)+\delta+v_+(x), \quad u_0(x)\ge u_{0-}(x)=p(x)-\delta+v_-(x).
\end{equation}
Notice that $p(x)\pm\delta$ are periodic functions with the same group of periods $G$ as $p(x)$ and with the mean values
$m\pm\delta\in (\alpha_-,\alpha_+)$. Let $u_+(t,x)$ be the largest e.s. of problem (\ref{1}), (\ref{2}) with initial function
$u_{0+}(x)$, and $u_-(t,x)$ be the smallest  e.s. of this problem with initial data
$u_{0-}(x)$. In view of (\ref{9}), we have $u_-(t,x)\le u(t,x)\le u_+(t,x)$. As we have already established, the e.s.
$u_\pm(t,x)$ satisfy relation (\ref{8}):
$$
\esslimsup_{t\to +\infty} \|u_\pm(t,\cdot)-(m\pm\delta)\|_X\le 2^{n+2}(\alpha^+-\alpha^-).
$$
This implies that
\begin{equation}\label{10}
\esslimsup_{t\to +\infty} \|u_\pm(t,\cdot)-m\|_X\le 2^{n+2}(\alpha^+-\alpha^-)+2^n\delta,
\end{equation}
where we use again the fact that measure of a unit ball in $\R^n$ is not larger than $2^n$.
Since the e.s. $u$ is situated between $u_-$ and $u_+$, we derive from (\ref{10}) that
$$
\esslimsup_{t\to +\infty} \|u(t,\cdot)-m\|_X\le 2^{n+3}(\alpha^+-\alpha^-)+2^{n+1}\delta,
$$
and to complete the proof it only remains to notice that a sufficiently small $\delta>0$ is arbitrary.
\end{proof}
Under the assumption of Theorem~\ref{thM} the value $\alpha^+-\alpha^-$ in (\ref{mest}) may be arbitrarily small. Therefore, (\ref{dec}) follows from (\ref{mest}). This completes the proof of our main Theorem~\ref{thM}.
Remark that in the case when the perturbation $v\ge 0$ ($v\le 0$) we can weaken the nonlinearity-diffusivity assumption in Theorem~\ref{thM} by the requirement $\forall b>m$ $(m,b)\cap F\not=\emptyset$ (respectively, $\forall a<m$ $(a,m)\cap F\not=\emptyset$). However, it is not possible, to weaken this assumption by the condition $m\in F$, as in the periodic case $v\equiv 0$. In the hyperbolic case $a\equiv 0$ this was confirmed by \cite[Example~2.6]{PaJHDE}. In the next section we construct another example for the parabolic equation $u_t-A(u)_{xx}=0$.

\section{Exactness of the nonlinearity-diffusivity assumption for a parabolic equation.}\label{ex}
In the one-dimensional case $n=1$ we consider the following purely parabolic equation
\begin{equation}\label{Df1}
u_t-A(u)_{xx}=0
\end{equation}
with the diffusion function $A(u)=u^+=\max(0,u)$, so that $a(u)=A'(u)$ is the Heaviside function (in fact, the problem (\ref{Df1}), (\ref{2}) is the Stefan problem). Since the flux vector is absent, the set $F$ in the nonlinearity-diffusivity condition consists of such points $u_0$ that
$A(u)$ is not constant in any vicinity $|u-u_0|<\delta$ of $u_0$. It is clear that $F=[0,+\infty)$. First, we are going to construct a $x$-periodic e.s. $u(t,x)$. It is known that e.s. of equation (\ref{Df1}) is characterized by the Carrillo conditions \cite{Car}: $A(u)_x\in L^2_{loc}(\Pi)$, and $\forall k\in\R$
\begin{equation}\label{DfE}
|u-k|_t-(\sgn(u-k)A(u)_x)_x=|u-k|_t-|A(u)-A(k)|_{xx}\le 0 \mbox{ in } \D'(\Pi).
\end{equation}
For piecewise $C^2$-smooth e.s the above conditions imply the following restrictions on discontinuity lines $x=x(t)$ similar to the jump conditions in \cite[Theorem~1.1]{VH}:
\begin{align}\label{Df2a}
[A(u)]=0, \\ \label{Df2b}
\forall k\in\R \quad [|u-k|]\nu_0-[\sgn(u-k)A(u)_x]\nu_1\le 0,
\end{align}
where $[w]=w_+-w_-$, $w_\pm=w_\pm(t)=w(t,x(t)\pm)$, is the jump of a function $w=w(t,x)$ across the line $x=x(t)$,
and $\nu=(\nu_0,\nu_1)$ is a normal vector on this line, directed from $w_-$ to $w_+$ (we may take $\nu=(-x'(t),1)$).
Conditions (\ref{Df2a}), (\ref{Df2b}) follow from entropy relation (\ref{DfE}) by application to a nonnegative test function and integration by parts (with the help of Green's formula). It is not difficult to verify that these conditions  together with the requirement that
$u$ is a classical solution in domains of its smoothness are equivalent to the statement that $u$ is an e.s. of (\ref{Df1}).
Let us make condition (\ref{Df2b}) more precise. Taking $k>\max(u_-,u_+)$, $k<\min(u_-,u_+)$, we readily arrive at the Rankine-Hugoniot relation
\begin{equation}\label{DfRG}
[u]\nu_0-[A(u)_x]\nu_1=0.
\end{equation}
Assuming (for fixed $t$) that $u_+>u_-$ and putting relation (\ref{Df2b}) together with (\ref{DfRG}), we arrive at the inequality
$$
(u_+-k)\nu_0 - (A(u)_x)_+\nu_1\le 0 \quad \forall k\in [u_-,u_+].
$$
This inequality is valid whenever it holds in the end-points $k=u_\pm$. Taking $k=u_+$, we obtain $(A(u)_x)_+\ge 0$. For $k=u_-$ we have $(u_+-u_-)\nu_0 -(A(u)_x)_+\nu_1\le 0$, and subtracting (\ref{DfRG}), we obtain that $(A(u)_x)_-\ge 0$, Similarly, in the case $u_+<u_-$ we derive the inequalities  $(A(u)_x)_\pm\le 0$. Hence,
\begin{equation}\label{DfEI}
\sgn(u_+-u_-)(A(u)_x)_\pm\ge 0
\end{equation}
The obtained conditions are different from the jump conditions of \cite[Theorem~1.1]{VH}. This connected with the fact that our diffusion function $A(u)$ is merely Lipschitz and the derivative $A(u)_x$ may have nontrivial traces at a discontinuity line.

To construct the desired solution, we need to solve the following initial-boundary value problem for the heat equation
$u_t=u_{xx}$ in the domain $t>0$, $|x|<r(t)=2-e^{-\alpha t}$, where the positive constant $\alpha$ will be indicated later, with homogeneous Dirichlet condition $u(t,\pm r(t))=0$ and with the initial condition $u(0,x)=\varphi(x)\in C_0^\infty((-1,1))$, where the initial function is even, $\varphi(-x)=\varphi(x)$.

Making the change $y=x/r(t)$ and denoting $v=v(t,y)=u(t,yr(t))$, we reduce our problem to the standard problem in the fixed segment $|y|\le 1$
\begin{equation}\label{Df3}
v_t=\frac{1}{r^2}v_{yy}+\frac{r'}{r}yv_y, \quad v(t,\pm 1)=0, \ v(0,y)=\varphi(y).
\end{equation}
Notice that the coefficient $1/4<\frac{1}{r^2}\le 1$ and in correspondence with the classic results \cite{LSU}, this problem admits a unique classical solution $v(t,y)$. Since $v(t,-y)$ is a solution of the same problem then, by the uniqueness, $v(t,-y)=v(t,y)$. Moreover, this solution is bounded by the maximum principle. Differentiating (\ref{Df3}) with respect to the space variable $y$, we derive that the functions $w=v_y$ and $p=v_{yy}$ satisfies the equations
\begin{align}\label{Df3d1}
w_t=\frac{1}{r^2}w_{yy}+\frac{r'}{r}(yw_y+w), \\
\label{Df3d2}
p_t=\frac{1}{r^2}p_{yy}+\frac{r'}{r}(yp_y+2p).
\end{align}
It follows from equations (\ref{Df3}), (\ref{Df3d1}) and the boundary condition $u(t,\pm 1)=0$ that at the boundary points $y=\pm 1$
\begin{align}\label{Df3b1}
p=v_{yy}=-rr'yv_y=-rr'yw, \\
\label{Df3b2}
\frac{1}{r^2}p_y=\frac{1}{r^2}w_{yy}=w_t-\frac{r'}{r}(yw_y+w)=w_t-\frac{r'}{r}(yp+w).
\end{align}
We multiply (\ref{Df3d2}) by $2p$ and integrate over the variable $y$. This yields the relation
\begin{equation}\label{Df4}
\frac{\partial}{\partial t}\int_{-1}^1 p^2dy=\frac{2}{r^2}\int_{-1}^1 pp_{yy}dy+\frac{r'}{r}\int_{-1}^1 (p^2)_yydy+\frac{4r'}{r}\int_{-1}^1 p^2dy.
\end{equation}
Integrating by parts, we obtain
$$
\int_{-1}^1 pp_{yy}dy=-\int_{-1}^1 (p_y)^2dy+pp_y|_{y=-1}^{y=1}, \quad \int_{-1}^1 (p^2)_yydy=-\int_{-1}^1 p^2dy+p^2y|_{y=-1}^{y=1}.
$$
Placing these equalities into (\ref{Df4}) and taking into account relations (\ref{Df3b2}), we arrive at the relation
\begin{align*}
\frac{\partial}{\partial t}\int_{-1}^1 p^2dy=-\frac{2}{r^2}\int_{-1}^1 (p_y)^2dy+2(w_t-\frac{r'}{r}w)p|_{y=-1}^{y=1}-\frac{2r'}{r}p^2y|_{y=-1}^{y=1}+ \frac{3r'}{r}\int_{-1}^1 p^2dy+ \\ \frac{r'}{r}p^2y|_{y=-1}^{y=1}=
-\frac{2}{r^2}\int_{-1}^1 (p_y)^2dy+\frac{3r'}{r}\int_{-1}^1 p^2dy-2rr'(w_t-\frac{r'}{r}w)wy|_{y=-1}^{y=1}-\frac{r'}{r}p^2y|_{y=-1}^{y=1}\le \\
-\frac{2}{r^2}\int_{-1}^1 (p_y)^2dy+\frac{3r'}{r}\int_{-1}^1 p^2dy-rr'\sum_{y=\pm 1}(w^2)_t(t,y)+2(r')^2\sum_{y=\pm 1}w^2(t,y),
\end{align*}
where we drop the non-positive term $\displaystyle-\frac{r'}{r}p^2y|_{y=-1}^{y=1}$.
Since $rr'(w^2)_t=(rr'w^2)_t-(rr')'w^2=(rr'w^2)_t-(rr''+(r')^2)w^2$, this relation can be written as
\begin{align}\label{Df6}
\frac{\partial}{\partial t}\left(\int_{-1}^1 p^2dy+rr'\sum_{y=\pm 1}w^2(t,y)\right)\le
-\frac{2}{r^2}\left(\int_{-1}^1 (p_y)^2dy+ \sum_{y=\pm 1}p^2(t,y)\right)+\nonumber\\ \frac{2}{r^2}\sum_{y=\pm 1} p^2(t,y)+\frac{3r'}{r}\int_{-1}^1 p^2dy+(rr''+3(r')^2)\sum_{y=\pm 1}w^2(t,y)\le\nonumber\\
-\frac{2}{r^2}\left(\int_{-1}^1 (p_y)^2dy+\sum_{y=\pm 1}p^2(t,y)\right)+\frac{3r'}{r}\int_{-1}^1 p^2dy+
5(r')^2\sum_{y=\pm 1}w^2(t,y),
\end{align}
where we used that $r''=-\alpha^2e^{-\alpha t}<0$ and that $\displaystyle\frac{2}{r^2}\sum_{y=\pm 1} p^2(t,y)=2(r')^2\sum_{y=\pm 1}w^2(t,y)$, by (\ref{Df3b1}).
Now we apply the Poincare inequality
$$
\int_{-1}^1 p^2dy\le c\left(\int_{-1}^1 (p_y)^2dy+\sum_{y=\pm 1}p^2(t,y)\right), \ c=\const,
$$
and derive
\begin{align}\label{Df7}
\frac{\partial}{\partial t}\left(\int_{-1}^1 p^2dy+rr'\sum_{y=\pm 1}w^2(t,y)\right)\le\nonumber\\
\left(-\frac{2}{cr^2}+\frac{3r'}{r}\right)\int_{-1}^1 p^2dy+5(r')^2\sum_{y=\pm 1}w^2(t,y)=\nonumber\\
\left(-\frac{1}{cr^2}+\frac{3r'}{r}\right)\int_{-1}^1 p^2dy-\frac{1}{cr^2}\int_{-1}^1 p^2dy+5(r')^2\sum_{y=\pm 1}w^2(t,y).
\end{align}
Since $w=u_y$ is an odd function and $p=w_y$, then $w(t,0)=0$ and $w(t,y)=\int_0^y p(t,s)ds$. By Jensen's inequality, this implies the estimate
$(w(t,y))^2\le\int_0^y (p(t,s))^2ds$ for all $y\in [-1,1]$. In particular,
$$\sum_{y=\pm 1}w^2(t,y)\le \int_{-1}^1 (p(t,y))^2dy$$ and it follows from (\ref{Df7}) that
\begin{align}\label{Df8}
\frac{\partial}{\partial t}\left(\int_{-1}^1 p^2dy+rr'\sum_{y=\pm 1}w^2(t,y)\right)\le\nonumber \\
-\left(\frac{1}{cr^2}-\frac{3r'}{r}\right)\int_{-1}^1 p^2dy-\left(\frac{1}{cr^2}-5(r')^2\right)\sum_{y=\pm 1}w^2(t,y).
\end{align}
Observe that $\frac{1}{cr^2}>\frac{1}{4c}$ while $\frac{3r'}{r}\le 3\alpha$, $5(r')^2\le 5\alpha^2$. Therefore, we can choose $\alpha>0$ so small that $\frac{1}{cr^2}-\frac{3r'}{r}>3\alpha$, $\frac{1}{cr^2}-5(r')^2>3\alpha$, $rr'<2\alpha<1$. Then, it follows from (\ref{Df8}) that
\begin{align*}
\frac{\partial}{\partial t}\left(\int_{-1}^1 p^2dy+rr'\sum_{y=\pm 1}w^2(t,y)\right)\le -3\alpha\left(\int_{-1}^1 p^2dy+rr'\sum_{y=\pm 1}w^2(t,y)\right),
\end{align*}
which implies the estimate
\begin{equation}\label{Df9}
\int_{-1}^1 p^2dy+rr'\sum_{y=\pm 1}w^2(t,y)\le Ce^{-3\alpha t}, \ C=\const.
\end{equation}
Since $rr'=\alpha re^{-\alpha t}>\alpha e^{-\alpha t}$, we conclude that
\begin{equation}\label{Df10}
|v_y(t,y)|=|w(t,y)|\le C_1e^{-\alpha t}\le C_2 r'(t) \ \forall t>0, y=\pm 1, \quad C_1,C_2=\const.
\end{equation}
It also follows from (\ref{Df9}) that
\begin{equation}\label{Df11}
v(t,y)\le\const\left(\int_{-1}^1 p^2dy\right)^{1/2}\le\const\cdot e^{-3\alpha t/2}\mathop{\to}_{t\to +\infty} 0.
\end{equation}
\medskip
We are going to construct a periodic e.s. of (\ref{Df1}) with period $5$ such that on the segment $[-5/2,5/2]$ it has the following structure: $u(t,x)$ is the described above solution of the Cauchy-Dirichlet problem for the heat equation in the domain $|x|<r(t)$, so that $u(t,x)=v(t,x/r(t))$, where $v(t,y)$ is the solution of problem (\ref{Df3}). For $r(t)<|x|<5/2$
our e.s. $u=u(x)\le 0$ satisfies the equations $u_t=0$, for $1<|x|<2$ the values $u(x)=-\psi(|x|)$ uniquely determined by Rankine-Hugoniot relation (\ref{DfRG}) on the line $x=r(t)$: $\psi(r)r'+u_x(t,r)=0$, so that
$\psi(r(t))=-u_x(t,r(t))/r'(t)=-v_y(t,1)/(r(t)r'(t))>0$. In view of estimate (\ref{Df10}) $\psi(x)\in L^\infty((1,2))$.
Finally, in the strips $2<|x|<5/2$ we set $u=0$. By the periodicity, the e.s. $u(t,x)$ is then extended on the whole half-plane $\Pi$, see figure~\ref{fig1}. By the construction, $u(t,x)$ satisfies all requirements (\ref{Df2a}), (\ref{DfRG}), (\ref{DfEI}) on the discontinuity lines $x=\pm r(t)$, Hence, it is an e.s. of the Cauchy problem for equation (\ref{Df1}) with $5$-periodic initial data
$$
u(0,x)=p(x)=\left\{\begin{array}{lcr} \phi(x) & , & |x|<1, \\ -\psi(|x|) & , & 1<|x|<2, \\ 0 & , & 2<|x|<5/2. \end{array}\right.
$$
Notice that
\begin{align*}
2\int_1^2\psi(x)dx=2\int_0^\infty\psi(r(t))r'(t)dt=-\int_0^\infty (u_x(t,r(t))-u_x(t,-r(t)))dt=\\ -\int_0^\infty dt \int_{|x|<r(t)}u_{xx}(t,x)dx=-\int_{-2}^2 dx\int_{|x|<r(t)}u_t(t,x)dt=\int_{-1}^1\phi(x)dx,
\end{align*}
where we use that $u=0$ on the set $|x|=r(t)$. This equality implies that
$$m=\frac{1}{5}\int_{|x|<5/2} u(0,x)dx=\frac{1}{5}\left(\int_{-1}^1\phi(x)dx-2\int_1^2\psi(x)dx\right)=0.$$
Using (\ref{Df11}), we conclude that $u(t,x)$ satisfies the decay property
$$
\lim_{t\to+\infty} u(t,\cdot)=0 \ \mbox{ in } L^1((-5/2,5/2)).
$$
This is consistent with the statement of Theorem~\ref{thPD} since $m=0\in F$. But $(a,0)\cap F=\emptyset$ for each $a<0$ and the nonlinearity-diffusivity condition of Theorem~\ref{thM} is violated. Let us show that after a perturbation $p(x)+v(x)$, $v(x)\in L^\infty_0(\R)$, of initial data the decay property can fail. In fact, taking the perturbation $v(x)<0$ supported on the segment $[2,3]$, we find that the corresponded e.s.
is $u(t,x)+v(x)$ and it does not satisfy the decay property, see figure~\ref{fig2}.

\begin{figure}[h!]
\includegraphics[width=3in]{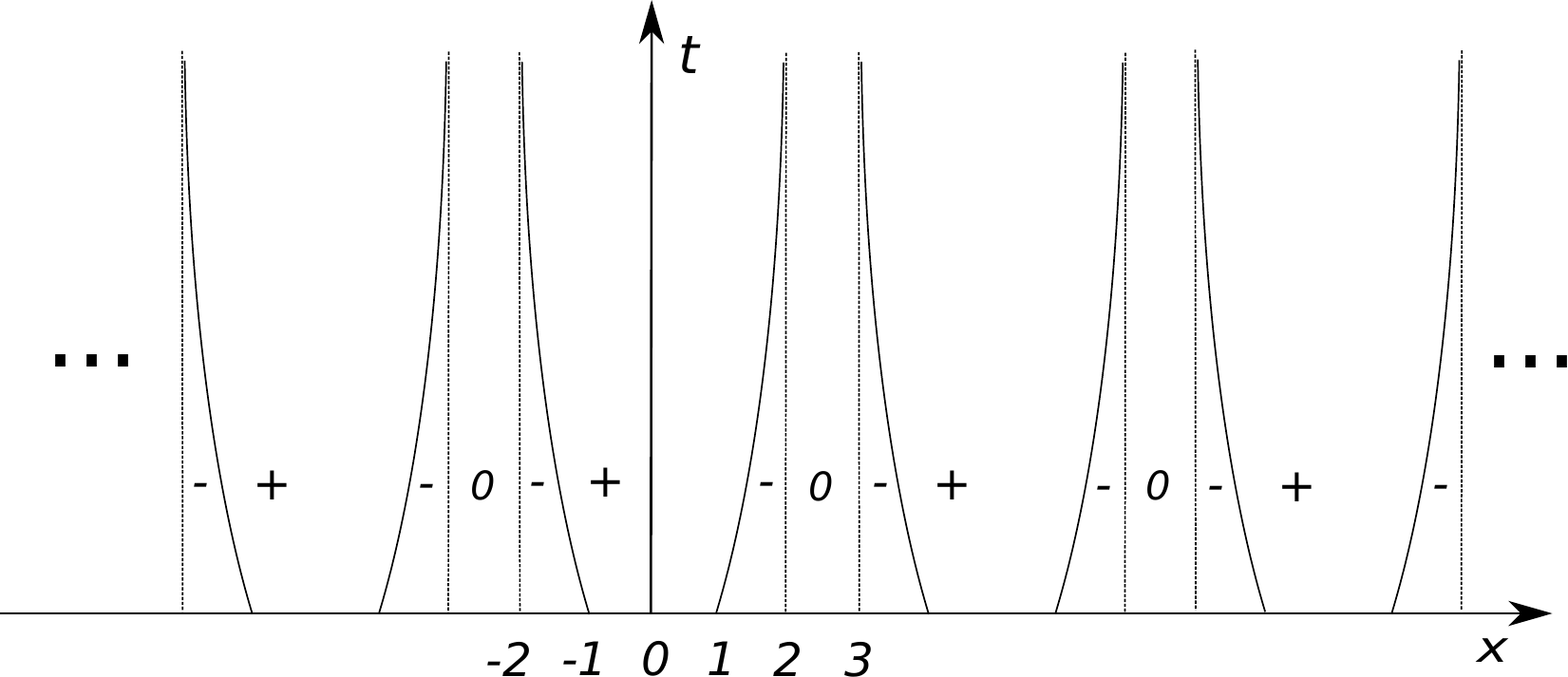}
\vspace*{4pt}
\caption{The solution $u(t,x)$ with the periodic initial data.
\label{fig1}}
\end{figure}

\begin{figure}[h!]
\includegraphics[width=3in]{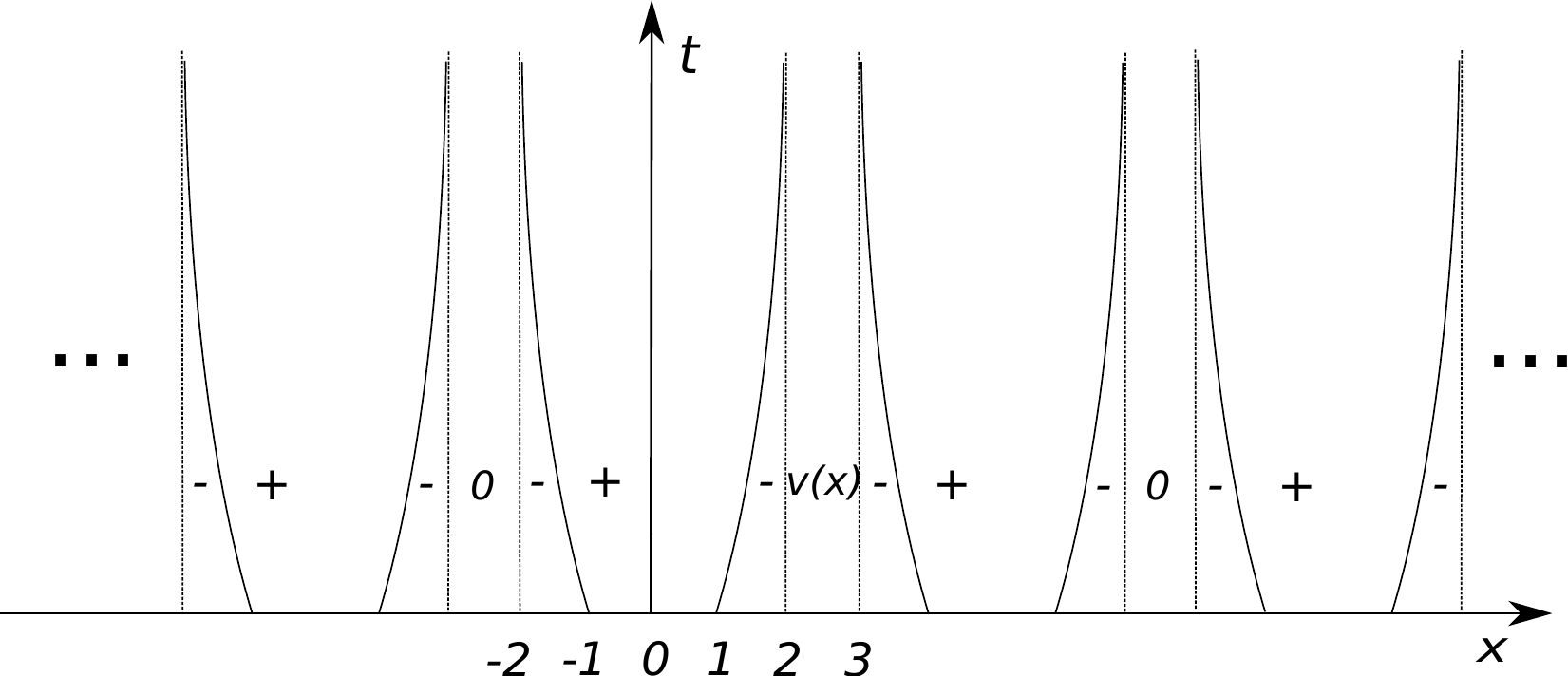}
\vspace*{4pt}
\caption{The solution $u(t,x)$ with the perturbed initial data.
\label{fig2}}
\end{figure}

\section*{Acknowledgments}\
This work was supported by the Russian Science Foundation, grant 22-21-00344.

\end{document}